\documentclass[11pt]{article}\usepackage{amsmath, amscd, amsthm} 
\usepackage{amssymb, amsfonts}
\input xy \xyoption{all}
\CompileMatrices
\numberwithin{equation}{subsection}%

\def\proof{\par\medskip\noindent {\bf Proof: }}
\def\qed{\hfill $\Box$ \medskip \par}
\def\Box{\framebox[10pt]{\rule{0pt}{3pt}}}

\ifx\ThmNum\undefined 
   \newtheorem{theorem}{Theorem}[section]
   
\else 
   \newtheorem{theorem}{Theorem} 
   
\fi

\newcommand{\R}{\mathbb{R}}

\newtheorem{remark}[theorem]{Remark}
\newtheorem{proposition}[theorem]{Proposition}
\newtheorem{lemma}[theorem]{Lemma}
\newtheorem{definition}[theorem]{Definition}
\newtheorem{corollary}[theorem]{Corollary}

\title{The homotopy automorphisms of a marked n-stage}
\author{ Manfred Stelzer}
\date{\today}

\begin{document}

\maketitle

\begin{abstract}
We show nilpotency and completeness results for the homotopy
automorphisms of a marked n-stage for an unstable coalgebra. These objects figure in the moduli problem of unstable coalgebras. Our theorems extend   classical work of Dror, Zabrodsky, Maruyama and M\o ller.
\end{abstract}

\section{Introduction}
\subsection{Overview}
 During the last decades moduli problems in homotopy theory became an important subject. In the realm of derived algebraic geometry various classical moduli problems of algebraic geometry where enhanced and extended to derived moduli problems in simplicial algebras or $E_{\infty}$-algebras. On the other hand
   realization  problems for classical topological invariants such as homotopy or (co)homology groups give rise to moduli problems as well.
 The case of rational cohomlogy was studied in the  eighties by various researchers in rational homotopy theory \cite{SSt},\cite{F},\cite{HSt}. Moduli spaces of realizations of $\pi$-algebras are the subject of \cite{BDG}. \\   
In this paper we continue the investigation of   the moduli space $\mathcal{M}_{Top}(C)$ of realizations of an  unstable coalgebra
$C$ which was started in \cite{BRS}. In order to set the scene we recall some structural results about these moduli spaces.
There is a  decomposition 
\begin{equation}\label{dec}\mathcal{M}_{Top}(C)\simeq \sqcup BAut^h (X)
\end{equation}
  where  $X$ runs through the topological realizations of $C$. 
Here $Aut^h(X)$ is the simplicial monoid of self homotopy equivalences of $X$ in a simplicial   model category.
 For $1$-connected $C$ the space $\mathcal{M}_{Top}(C)$ is fibered in a convergent tower with equivalences
 \begin{equation}\label{tow}\mathcal{M}_{Top}(C)\simeq  holim \mathcal{M}_{n}(C)
\end{equation}  

 where  the spaces $\mathcal{M}_{n}(C)$ are moduli spaces attached to cosimplicial truncated versions of realizations for the given coalgebra. 
This tower gives a way to do obstruction theory on the realization problem as the homotopy fibers of the maps between layers have a cohomological description.  The truncated objects of level $n$ are called potential $n$-stages. Given a potential $(n-1)$-stage $X$ for $C$ which extends to a potential $n$-stage let  $\mathcal{M}_{n}(C)_{X}$ be  the subspace of $\mathcal{M}_{n}(C)$ over the component of $X$. 

  The set of path components  can  be  described as the following set of orbits    \cite[proposition 7.4.2.]{BRS}. It counts the extensions of the given $(n-1)$-stage $X$ to an $n$-stage.
\begin{equation}\label{moti}
\pi_0(\mathcal{M}_{n}(C)_{X}) \cong \frac{AQ^{n+1}_C(C; C[n])}{\pi_1(B Aut^h(X))}.\end{equation}
 
Here $AQ^{n+1}_C(C; C[n])$ is the Andr\'{e}-Quillen cohomology of the unstable coalgebra $C$ with coefficients in the shifted comodule $C[n]$ underlying $C$.  This Andr\'{e}-Quillen cohomology 
makes up the $E_2$-term of an unstable Adams spectral sequence.\\
The goal of this paper is to obtain structural nilpotency and completeness results about the simplicial monoid  $Aut^h(X)$ and its group of path components $ \pi_0 Aut^h(X)$.


 \subsection{Statement of results}
    Let  $\mathbb{F}$ a prime field of characteristic $l$.  All homology groups $H_* (-)$ will  be taken with $\mathbb{F}$ coefficients throughout. In the following the term $l$-complete stands for $H_* (-)$-complete in the sense of \cite{Bou2}.\\
  Let $X$ be a potential n-stage for the unstable coalgebra $C$  and let $G$ be a subgroup of $\pi_0 Aut^h(X)$ which acts nilpotently on  \[\pi^0 H_* (X)\cong C.\]  Moreover, let $Aut^{h}_G (X)$ be the submonoid of $Aut^h(X)$   of all components of $G\subset \pi_0 Aut^h(X)$. We write $\tilde{G}$ for the group of automorphisms of $C$ induced by $G$. The group $G_{\sharp}$ of homotopy classes of automorphisms which induce the identity on $C$ is worth special attention. For short we write $Aut^{h}_{\sharp} (X)$ for $Aut^{h}_{G_{\sharp}} (X)$.\\
   \\
    Our main result reads:
   
\begin{theorem}\label{main}  Let $X$ and $G$ be as above.
Then \begin{enumerate}
\item $BAut^{h}_G (X)$ is a nilpotent space.
\item $BAut^{h}_G (X)$ is $l$-complete in case  $G$ contains $G_{\sharp}$ and $\tilde{G}$ is $l$-complete.\end{enumerate}  \end{theorem}

Since the monoid $Aut^{h}_G (X)$ is group like we obtain:

\begin{corollary}\label{maincor} Let $X$ and $G$ be as above. Then the  group $\pi_0 Aut^{h}_G (X)$ is nilpotent and $l$-complete.\end{corollary}


\begin{corollary}\label{maincor3} Let $X$  be as above. Then the  space $B Aut^{h}_{\sharp} (X)$ is nilpotent and $l$-complete.\end{corollary}

Here is another consequence.

\begin{corollary}\label{maincor4}Let   $X$ be an $l$-complete realization of a $1$-connected
 $C$.
   Then $BAut^h_{\sharp} (X)$ is $l$-complete. More precisely, it is a homotopy limit of $l$-complete nilpotent spaces.\end{corollary}


Our results are analogues of the classical theorems of Dror, Zabrodsky, 
 Maruyama and M\o ller. 
 These authors studied  the group of homotopy self equivalences of a finite complex or a finite Postnikov section $X$ which induce the identity on homotopy or homology. They showed nilpotency \cite{DZ} and $\mathbb{F}$-completeness of the group of path components for $\mathbb{F}$-complete $X$ \cite{Ma},\cite{Mo}. More recently, some of these results were extended for p-local homotopy groups and homotopy groups  with mod-p coefficients by Cuvilliez, Murillo and Viruel \cite{CMV}.

The proof of \ref{main} is by induction on the skeletal filtration of an $n$-stage $X$. It was shown in \cite{BRS} that this filtration  is manufactured by homotopy pushouts where the attached spaces are  twisted (dual) Eilenberg-MacLane objects which corepresent Andr\'{e}-Quillen cohomology.   Following a strategy pioneered by  Dror and Zabrodsky,   we show that the group $G$ operates nilpotently on Andr\'{e}-Quillen cohomology of $C$.  While Dror and Zabrodsky  employ the Serre spectral sequence we use a variant of an inverse Adams  spectral sequence due to Miller whose $E_2$-term depends only on  $\pi^* H_* (X)$. As a substitute for the application of classical obstruction theory in \cite{DZ} we apply the main results on the structure of the moduli spaces $\mathcal{M}_{n}(C)$ in \cite{BRS}.\\

\begin{remark} In rational homotopy theory there are related results known \cite{SSt},\cite{F},\cite{HSt}. However, the technology used in these references is   different
 from ours. For example,  the notion of an $n$-stage does not show up. But prounipotency results are proved for groups of homotopy self equivalences of a rational space by means of certain truncations of a filtered model. It would be interesting to pin down the exact relation to our results for the prime field of rationals. However, this will not be an easy task, as the moduli object is represented by a differential graded coalgebra or Lie algebra which has non-nilpotent path components in general.\\\end{remark}

\subsection{Organization of the paper.} In Section 2 we briefly recollect facts about the relevant resolution model categories. We also give a reminder  on homotopy invariants such as  Andr\'{e}-Quillen cohomology of unstable coalgebras. Section 3  collects  and complements
information on moduli spaces of $n$-stages and on certain variations which rigidify  the moduli problems. These moduli spaces are connected by fibration sequences which are used in the inductive arguments later on. Section 4 is devoted to the construction of a spectral sequence which converges to Andr\'{e}-Quillen cohomology. This spectral sequence  is used in Section 5 to deduce nilpotency of the $G$-action on Andr\'{e}-Quillen cohomlogy. In Section 6 we put things together and prove the main theorem \ref{main}. In an appendix we establish results  about coalgebras over a certain comonad. These are needed in the construction of the spectral sequence in Section 4. The proofs depend heavily on $\infty$-categorical technology.\\

\subsection{Conventions}   We will use the following notation. 
\begin{itemize}
\item $\mathcal{S}$ = the category of simplicial sets;
\item $Vec$ = the category of $\mathbb{F}$-vector spaces
\item  $\mathcal{CA}$ = the category of unstable coalgebras over the Steenrod algebra $\mathcal{A}$ over $\mathbb{F}$
\end{itemize}
For any category $\mathcal{C}$ we let
\begin{itemize}
\item  $c\mathcal{C}$ = the category of cosimplicial objects over  $\mathcal{C}$;
\item  $n\mathcal{C}$ = the category of non-positively graded objects over $\mathcal{C}$.
\end{itemize}

\subsection{Acknowledgments}I heartily thank Hadrian Heine who provided me with the proofs of  all the $\infty$-categorical results in the appendix. I have benefited  from conversations related to the subject of this article with Markus Spitzweck and at an earlier stage with Georg Biedermann and George Raptis. Finally, I would like to thank the Deutsche Forschungsgemeinschaft for
support through the Schwerpunktprogramm 
1786 ``Homotopy theory and algebraic geometry''.
 


\section{Cosimplicial spaces and unstable coalgebras} \label{Res}


\subsection{Resolution model categories}
The categories $\mathcal{S}$  and $\mathcal{CA}$  are equipped with the standard Quillen and the discrete model structure respectively.  
 Let $K(\mathbb{F},m)$ denote the Eilenberg-MacLane space of type $\mathbb{F}$ indegree $m$. A product of these spaces will be called an $\mathbb{F}$-Gem. Let $\mathcal{G}=\{K(\mathbb{F},m)| m\geq 0 \}$ 
 and $\mathcal{E}=\{H_* (K(\mathbb{F},m))| m\geq 0\}.$  The categories $c\mathcal{S}$ of cosimplicial spaces and cosimplicial unstable coalgebras $c\mathcal{CA}$  carry simplicial  resolution model category structures relative
to $\mathcal{G}$ and $\mathcal{E}$ respectively \cite{Bou4},\cite{BRS}.\\
The weak equivalences are the maps of cosimplicial objects 
\begin{equation}X\to Y\end{equation} which induce an isomorphism on $\pi_* [-,F]$ for all $F$ in $\mathcal{G}$ or $\mathcal{E}$ respectively. This can be equivalently described as maps which induce an isomorphism on $\pi^* H_* (-)$ and $\pi^* (-)$ respectively. For  $\mathcal{E}$ this equivalence is due to the fact that objects in $\mathcal{E}$   are cofree unstable coalgebras. 

 The cofibrations are the Reedy cofibrations such that the induced homomorphism\begin{equation} [Y,F]\to[X,F]\end{equation} is a fibration of simplicial groups. A concrete description of the fibrations  can be found in \cite[2.3.]{BRS}. We write $c\mathcal{S}^{\mathcal{G}}$ and $c\mathcal{CA}^{\mathcal{E}}$ for these model categories.
The singular homology functor connects these two model model categories. 
 
 \subsection{Andr\'{e}-Quillen cohomology}
Let $\mathcal{V}$ be the full subcategory of all unstable right $\mathcal{A}$-modules $M$ which satisfy the strong unstability condition
\begin{align*}
  x P^n &= 0  \text{ for } |x|\le 2pn \text{ and } p \text{ odd, } \\
 x Sq^n &= 0 \text{ for } |x|\le 2n \text{ and } p=2. 
\end{align*}
For $C\in \mathcal{CA}$ we let $\mathcal{V}C$ stand for the $C$-comodules in $\mathcal{V}$.
 The category $\mathcal{V}C$ is equivalent to the category of coablelian cogroup objects $C/\mathcal{CA}_{ab}$  in $C/\mathcal{CA}$. For $M\in \mathcal{V}C$ let $i_CM=C\oplus M$ with coalgebra structure defined by the comodule structure.
 There is  an adjunction \cite{Bou3}
\begin{equation}\iota_C :  \mathcal{V}C \rightleftarrows C/\mathcal{CA}: Ab_C \end{equation} 

 \begin{definition}
 Let $D\in  C/ c\mathcal{CA}$ and $M\in \mathcal{V}C$ and $M[n]$ the $n$-fold shift of $M$. 
 The $n$-th Andr\'{e}-Quillen cohomology of $D$ with coefficients in $M$ is defined as 
 \[AQ^n_C (D, M)= [M[n],\mathbf{R}Ab_C (D)]_{\mathcal{V}C}=\pi^n (Hom_{\mathcal{V}C}(M,\mathbf{R}Ab_C (D) ))\]
 \end{definition}
 
For $M \in \mathcal{V} C$  there are twisted Eilenberg-MacLane objects 
 $L_{C} (M,n)$, $K_{C} (M,n)$ \cite[5.3, 6.3]{BRS}
   which corepresent  Andr\'{e}-Quillen cohomology in  $c\mathcal{S}^{\mathcal{G}}$ and $c\mathcal{CA}^{\mathcal{E}}$ respectively.
 
\begin{definition}Let $M \in \mathcal{V} C$ and $D \in C/\mathcal{CA}$. The Andr\'{e}-Quillen space of $D$ with coefficients in $M$ is defined as
 
 \[ \mathcal{AQ}^n_{C}(D;M) 
  = map^{ der}_{c(C/ \mathcal{CA})}(K_C(M,n), cD).\]\\
\end{definition}
The homotopy groups of these spaces satisfy:
 
 \begin{equation}\label{AQHG}   \pi_s \mathcal{AQ}^n_C(D;M)\cong\left\{\begin{array}{cl}
                          AQ^{n-s}_C(D;M)& 0\le s\le n \\
                                0       & \text{otherwise}
                                 \end{array}\right. 
\end{equation}

 There are equivalences 
 \[AQ^n_C (D, M)\simeq map^{der}_{L(C,0)}(L_C (M,n),Y)\simeq  map^{der}_{C}(K_C (M,n),H_* (Y)).
 \]
 Here we use the shorthand  $L(C,0)$   for $L_C (C,0)$.
 
 Automorphism groups of the coabelian objects  $\iota_C(M)$ will be important later on.
 
 \begin{definition}An automorphism of $\iota_C(M)$ is an isomorphism $\phi : \iota_C(M) \to \iota_C(M)$ which is compatible with an isomorphism $\phi_0$ 
of $C$ along the projection map:
\[
 \xymatrix{
 \iota_C(M) \ar[r]^{\phi}_{\cong} \ar[d] & \iota_C(M) \ar[d] \\
 C \ar[r]^{\phi_0}_{\cong} & C .
 }
\]
This group of automorphisms will be denoted by $Aut_C(M)$. 
\end{definition}

\subsection{The spiral exact sequence}
The spiral exact sequence relates two invariants whose definitions we recall now.
 \begin{definition} For $X\in c\mathcal{S}^{\mathcal{G}}$
  and $G$ an $\mathbb{F}$-Gem define:\\
  the $E_2$-(co)homotopy groups  as 
 \[\pi_* [X,G]\]
 and the natural (co)homotopy groups  as  
 \[\pi^{\sharp}(X,G)=[X,\Omega_{ext}^{s}(cG)]\]
 where $\Omega_{ext}^s$ is taken with respect to the external simplicial structure and $cG$ denotes the constant cosimiplicial object $G$.
 \end{definition}
 
 In order to describe the algebraic structure on the spiral exact sequence we need:
 
 \begin{definition}
  Let be  $\mathcal{H}_{un}$ the full subcategory of
   Ho$(\mathcal{S})$ given by finite products of
    $\mathcal{H}_{un}$-Eilenberg-MacLane spaces.\\ 
 An  $\mathcal{H}_{un}$-algebra is a product preserving functor from $\mathcal{H}_{un}$ to the category of sets.\\
  Replacing Ho$(\mathcal{S}$) by Ho$\mathcal{S}_* $ one arrives at the definition of an $\mathcal{H}$-algebra.
 \end{definition}

\begin{theorem}\cite[Theorem 3.3.1.]{BRS}\label{ses-statement}
There is an isomorphism of $\mathcal{H}_{un}$-algebras 
    $$ \pi_{0}^{\sharp}(X,G)\cong\pi_{0}[X,G]$$ 
and a long exact sequence of  $\pi_0$-modules and $\mathcal{H}$-algebras
\begin{align*}
    ... \to \pi_{s-1}^{\sharp}(X,\Omega G)\to \pi_{s}^{\sharp}(X, G)\to\pi_s[X,G]\to \pi_{s-2}^{\sharp}(X,\Omega G) \to ... \\
    ... \to \pi_{2}[X,G]\to \pi_{0}^{\sharp}(X,\Omega G)\to \pi_{1}^{\sharp}(X, G)\to \pi_1[X,G] \to 0,
\end{align*}
where $\Omega$ is the internal loop functor.
\end{theorem}

\begin{remark} An $\mathcal{H}_{un}$-algebra is the same thing as an unstable algebra \cite[Theorem B.3.3]{BRS}. The $\pi_0$-module structure on the terms  $\pi_s[X,G]$ corresponds to the usual action of the unstable algebra $\pi_0 H^* (X)$ on the unstable module $\pi_s H^* (X)$ by \cite[Corollary B.3.5]{BRS}.
\end{remark}

 
 \section{Moduli spaces of n-stages}  \label{Mod}For $C\in \mathcal{CA}$ let $\mathcal{M}_{Top}(C)$ be the moduli space of all topological realizations of $C$. It is defined as the classifying space of the category whose objects are the spaces with homology isomorphic to $C$ and whose morphisms are the homology equivalences. The space of marked realizations $\mathcal{M}^{'}_{Top}(C)$ of $C$ is a rigidification of  $\mathcal{M}_{Top}(C)$.
 It is defined as the classifying space of the category whose objects are realizations $X$ together with an isomorphism $\sigma:C\to H_* (X)$ in  
 $\mathcal{CA}$. Morphisms $f:(X;\sigma)\to (Y,\tau)$ are homology equivalences $f:X\to Y$ such hat $H_* (F)\sigma= \tau$ holds.
 
  In order to decompose $\mathcal{M}_{Top}$ one replaces spaces by resolutions in $c\mathcal{S}^{\mathcal{G}}$. The decomposition is defined by a Postnikov decomposition in the cosimplicial direction in  $c\mathcal{S}^{\mathcal{G}}$.  To make this precise we need.\\
 
 \begin{definition}\label{def:n-stage}
A cosimplicial space $X$ is called
 a potential $n$-stage for $C$ if 
\begin{itemize}
 \item[(a)]

$\pi^s H_*(X) \cong \left\{
                   \begin{array}{cl}
                       C      & s = 0 \\ 
                       C[n+1] & s=n + 2 \\
                       0      & \hbox{otherwise, }
                   \end{array}    \right.$

where the isomorphisms are between unstable coalgebras and $C$-comodules respectively.
 \item[(b)] $\pi^{\sharp}_{s}(X,G) = 0$ for $s > n$ and $G\in \mathcal{G}$.
  \end{itemize} 
 A cosimplicial space $X$ is called an $\infty$-stage for $C$ if \\
  $\pi^s H_*(X) \cong \left\{
                   \begin{array}{cl}
                       C      & s = 0 \\ 
                      
                       0      & \hbox{otherwise, }
                   \end{array} 
                                        \right.$\\
                                        
                                        A resolution for $C$ is a fibrant $\infty $-stage for $C$.
  \end{definition}
\begin{remark}If $X$ is a resolution of $C$ in $c\mathcal{S}^{\mathcal{G}}$, then the $n$-skeleton functor  $sk_n X$ is an $(n-1)$-stage. More generally for $m\geq n$, the derived $n$-skeleton of an $m$-stage is an $(n-1)$-stage.
\end{remark}
\

 \begin{definition}\label{def:marked n-stage}
 Fix a cofibrant object $L(C,0)$. A marked $n$-stage is an $n$-stage $X$ together with a morphism $\sigma:L(C,0)\to X$ which induces an isomorphism on $\pi^0$.
 \end{definition}
The following results \ref{23}-\ref{framed} can be found in \cite[Section 7.1-7.3]{BRS}.
For $n\leq \infty$ there are moduli spaces of  $n$-stages and of marked $n$-stages 
 $\mathcal{M}_{n}(C)$, $\mathcal{M}^{'}_{n}(C)$ defined as classifying spaces of the categories of $n$-stages and marked $n$-stages respectively. In  case $C$ is simply connected there are equivalences \cite[]{BRS}
 \begin{equation}\label{23}\mathcal{M}_{Top}(C)\simeq \mathcal{M}_{\infty}(C)\hspace{0.6cm}
 \mathcal{M}^{'}_{Top}(C)\simeq \mathcal{M}^{'}_{\infty}(C) \end{equation} In general there are  equivalences
 \begin{equation}\label{33} \mathcal{M}_{\infty}(C)\simeq holim \mathcal{M}_{n}(C)
\hspace{0.6cm} \mathcal{M}^{'}_{\infty}(C)\simeq holim \mathcal{M}_{n}^{'}(C). \end{equation}
For $n=0$ one has 
 \begin{equation}\label{0stage}\mathcal{M}_{0}(C)\simeq BAut(C)\hspace{0.6cm}
 \mathcal{M}^{'}_{0}(C)\simeq * \end{equation}
 
The spaces $\mathcal{M}_n(C)$ and $\mathcal{M}^{'}_n(C)$ are of the homotopy type of the disjoint union  $\sqcup BAut^h(X)$ where $X$ runs through all unmarked and marked n-stages for $C$ respectively. 
 
 There is a homotopy fiber sequence \cite[p.9,7.3.1.]{BRS} \begin{equation}\label{FS}\mathcal{M}^{'}_{n}(C)\to \mathcal{M}_{n}(C)\to BAut(C).\end{equation}

The derived $n$- skeleton functor defines homotopy fiber sequences
connecting the moduli spaces of $n$ and of $(n-1)$-stages. 

 

 

Let \begin{equation}\label{89}\tilde{\mathcal{AQ}}_C^{n+2}(C,C[n])=\mathcal{AQ}_C^{n+2}(C,C[n])// Aut_C(C[n])\end{equation} be the homotopy quotient by the action of $Aut_C(C[n])$.
There are  homotopy pullback squares \cite[p.99]{BRS}

\begin{equation}\label{15} 
  \xymatrix{
  \mathcal{M}'_n(C) \ar[d] \ar[r]^-{} & \mathcal{M}_n(C) \ar[d] \ar[r] & BAut_C(C[n]) \ar[d]\\
  \mathcal{M}'_{n-1}(C) \ar[r]^-{} & \mathcal{M}_{n-1}(C) \ar[r] & 
 \tilde{\mathcal{AQ}}_C^{n+2}(C,C[n]) }
 \end{equation}
 where the map on the right is defined by the zero Andr\'e-Quillen cohomology class. 
The fiber of this map is \[\Omega \mathcal{AQ}^{n+2}_C(C; C[n]) \simeq \mathcal{AQ}^{n+1}_C(C; C[n])). \]

Consequently, for each marked (n-1)-stage $Y$
 there is a homotopy fiber sequence fibre sequence of the form
 \begin{equation}\label{43}\mathcal{A}^{n+1}(C,C[n]) \to \mathcal{M}_n^{'}(C)_{Y}\to BAut_{\sharp}(Y)\end{equation}
 where $\mathcal{M}_n^{'}(C)_{Y}$ is the homotopy pullback of the map from $*$ which picks out the component  corresponding to $Y$.  Choose a base point in $\mathcal{M}_n^{'}(C)_{Y}$ that is fix  an $n$-stage $X$ which restricts to $Y$.\\
  The looped homotopy fiber sequence of \ref{43} is then of the form
  \begin{equation}\label{45}\mathcal{A}^{n}(C,C[n]) \to Aut^h_{\sharp}(X)\to Aut^h_{\sharp}(Y)\end{equation}
 Consider 
\begin{equation}\label{frameddef}\mathcal{M}'_n(C) \longrightarrow \mathcal{M}_n(C) \to B ( Aut_C(C[n+1]).\end{equation}
where the last map is given by the functor $\pi^{*} H_* (-)$.

Define the moduli space of framed potential $n$-stages $\mathcal{M}^f_n(C)$ to be the homotopy fiber of this composite map.   
The analog of \ref{0stage} is 
\begin{equation}\label{framed0}
\mathcal{M}^f_0(C) \simeq Aut_C(C[1]).
\end{equation}
There is a homotopy pullback square  \cite[p.100]{BRS}
 \begin{equation}\label{framed}
 \xymatrix{
\mathcal{M}^f_n(C) \ar[rr] \ar[d] && \ast \ar[d]^0 \\
\mathcal{M}^f_{n-1}(C) \ar[rr] && \mathcal{AQ}^{n+2}_C(C, C[n])}
\end{equation}

 

  
 

 \begin{lemma} Let $X$ be a marked $n$-stage for $C$ and $Z$ the underlying $n$-stage. There is an exact sequence
 \[1\to \pi_1 (BAut^h(X))\to \pi_1 (BAut^h(Z))\to Aut(C)\] 
 \end{lemma}
 \begin{proof} The assertion follows from the long exact homotopy sequence of \ref{FS}.$\square$
 \end{proof}
 
 \begin{corollary} The group $\pi_1 (BAut^h(X))\cong \pi_0 Aut^h(X)$ is the subgroup of homotopy classes in  $\pi_0 Aut^h(Z)$ which induce the identity on $C$.
 \end{corollary}
From \ref{45} we get 
 \begin{proposition}\label{long} Let $X$ be a marked $n$-stage for $C$ which restricts to the marked $(n-1)$-stage $Y$. Then there is a long exact sequence 
 \begin{equation}\ldots \to \pi_s \mathcal{AQ}^n_C(C;C[n])\to 
 \pi_s (Aut^h_{\sharp}(X))\to  \pi_s (Aut^h_{\sharp}(Y))\to \ldots
 \end{equation}
 \end{proposition}

\section{A coalgebraic Miller spectral sequence}

Let $G$ be the cofree unstable coalgebra functor right adjoint to the forgetful functor to graded vector spaces. 
For $C \in \mathcal{CA}$ there is an induced adjunction between the under-categories

\begin{equation}J: C / \mathcal{CA} \rightleftarrows  J(C) / nVec :G.\end{equation} 

This adjunction extends to a Quillen adjunction on cosimplical objects. Use \cite[4.3.4.]{BRS} to see this.
 Here the model categories are the resolution model categories 
 defined by the group objects $H_* K(\mathbb{F},n)$ and the standard model category respectively. 
 There is a natural isomorphism $G(V)\cong H_* (K(V))$ with
  $K(V)=\prod_{n\geq 0} K(V_n ,n).$
 The functor  $G$   defines a comonad  on   $cnVec$ still denoted by $G$.
 
 \begin{proposition} There is a comonad on bigraded vector spaces \[\mathcal{G}:nnVec \to nnVec\] such that \[\mathcal{G}(\pi^* (V))=\pi^* (G(V))\] for any $V\in cnVec$.
 Moreover, for $C\in c\mathcal{CA}$ the cohomology $\pi^* (C)$ is a coalgebra over  $\mathcal{G}$.
 \end{proposition}
 \begin{proof} This follows from the dual of \cite[5.1]{G} as in \cite[p.131-132]{G}. $\square$
 \end{proof}
 

\begin{definition}
 Let \begin{enumerate}
\item $\mathcal{G}_{coA}$ the category of coalgebras for the comonad $\mathcal{G}$.\\
\item   $coGr\mathcal{G}_{coA}$ be the category of coabelian cogroup objects in $coAlg\mathcal{G}$.\\

\end{enumerate}\end{definition}

\begin{remark} Any $C\in \mathcal{CA}$ viewed as a constant cosimplicial object defines an object in 
$\mathcal{G}_{coA}$.
\end{remark}


\begin{theorem}\label{adjoint} Let $C\in \mathcal{CA}$.
 under $C$.
The inclusion functor 
\[i:C/coGr\mathcal{G}_{coA}\to C/\mathcal{G}_{coA}\]

admits a right adjoint $Ab_C$
\end{theorem}

The proof by an  $\infty$-categorical argument can be found in the appendix. It is more involved than the one for the  counterpart for algebras.
In fact it can be shown that a construction   dual to the module of relative K\"ahler differentials gives $Ab_C$. The proof needs a detailed study of Bousfields higher divided power operations \cite{Bou1} and will be given in a later paper \cite{St}. 
 For the category of cocommutative coalgebras such a result is obtained  by  Slomińska  in \cite{Sm}. 

 
 
 \begin{remark} In \cite{M}, and \cite{G} a left adjoint was constructed in the dual situation of algebras. This played an important role  in the derivation of the inverse Adams spectral sequence which was instrumental in the proof of the Sullivan conjecture.
 \end{remark}

 
 \begin{theorem} Let $B\in c(C/\mathcal{CA})$ with  $C\in \mathcal{CA}$ and $M\in \mathcal{V}C$ .  There is a convergent spectral sequence 
\[E_2^{p,q}\cong L^p_{\mathcal{G}} (Hom_{C/coGr\mathcal{G}_{coA}}(i( M[q]),Ab_C (\pi^* (B )))\Longrightarrow AQ^{p+q}_C (B,M)\]\label{SS}
\end{theorem}

\begin{proof}There is a triple derived fibrant $G$-cofree resolution  of  $B$ in $c(C/\mathcal{CA})$. It derived from the   adjoint pair $G,J$ in the standard way by means of diagonalizing the bicosimplicial object
$G^{n+1}(B^m )$. Fibrancy follows from \cite{Bou}.  Thus in cosimplicial degree $n$ it is given as 
$G^{n+1}(B^n )$.  Now consider the bicosimplicial object $U$ with 
\[U^{p,q}= Hom_{C/\mathcal{CA}}(i(M),G^{p+1}(B^q ))\cong Hom_{\mathcal{V}C}(M,Ab_C (G^{p+1}(B^q )).\] Filter by $p$ to obtain a convergent spectral sequence \[\pi^p \pi^q (U)\Longrightarrow \pi^{p+q}(diag(U)).\]
Now since any fibrant resolution can be used to define $AQ^{p+q}_C(B,M)$ we find \[\pi^{p+q}(diag (U)= \pi^{p+q}(Hom_{\mathcal{V}C}(M,Ab_C 
(G^{*+1}(B^* ))=AQ^{p+q}_C(B,M).
\] so the spectral sequence converges to the right target.\\
We are left to determine the $E_2$ term.

\begin{align*}E^{p,q}_1 \cong \pi^q (Hom_{C/\mathcal{CA}}(i(M),G^{p+1}(B ))\\
\cong \pi^q (Hom_{J(C)/n Vec}(i(M),G^{p}(B ))\\
\cong  Hom_{J(C)/nVec}(i(M),\pi^q(G^{p}(B ))\\
\cong Hom_{J(C)/n Vec}(i( M),\mathcal{G}^p (\pi^* (B )^q)\\
\cong Hom_{C/\mathcal{G}_{coA}}(i( M[q]),\mathcal{G}^{p+1} (\pi^* (B )) \\
\cong Hom_{C/coGr\mathcal{G}_{coA}}(i( M[q]),Ab_C (\mathcal{G}^{p+1} (\pi^* (B )))
  \end{align*}
  
It follows that \[E^{p,q}_2\cong L^p_{\mathcal{G}} (Hom_{{C/coGr\mathcal{G}_{coA}}}(i( M[q]),Ab_C (\pi^* (B ))) \]$\square$
\end{proof}
\section{G-actions on Andr\'e-Quillen cohomology}
If a group $G$ acts on another group $A$ the $n$-th $G$-commutator group $\Gamma_G^n (A)\subset A $ is the group generated by all $\{[g,a]| g\in G, a\in\Gamma_G^{n-1}\}$ with $\Gamma_G^0 = A$ and $[g,a]=(ga)a^{-1}$. The action is nilpotent of nilpotence order $nil_G A = r$ if $r$ is the smallest integer such that $\Gamma_G^n (A)=1$.\\
We record a result of Hall for later use.  It is a vast generalization of the fact that a group of unipotent matrices is nilpotent and it was used already in \cite{DZ}. A proof can be found in \cite[p.355]{P}.

\begin{theorem}\label{Hall} Let $G$ be a group of automorphisms of another group $A$. Assume that $G$ acts nilpotently on $A$; then $G$ is nilpotent.
\end{theorem}
 
\begin{lemma}\label{nil1}Let $X$ be an $n$-stage for $C$ and $G\subset \pi_0 Aut^h (X)$ a subgroup. The isomorphism of $C$-comodules \[\pi^0 H_* (X)\cong \pi^{n+2} H_* (X)\] is $G$-equivariant. 

\end{lemma}
\begin{proof} The group $G$ acts on the spiral exact sequence 
\ref{ses-statement} of $X$.  The $C$-comodule automorphism  induced  by $g\in G$ on $\pi^{n+2}H_* (X)\cong C[n+1]$ is identified via the spiral exact sequence with the automorphism induced by $g$ on $\pi^{0}(C)\cong C$ \cite[7.1]{BRS}.$\square$
\end{proof}

\begin{lemma}\label{nil2}Let $X$ be in $L(C,0)/c\mathcal{S}$. Suppose $G$ acts nilpotently on $\pi^* H_* (X)$. Then $G$ acts nilpotenty on $Ab_C(\mathcal{G}^n (\pi^* (X)))$.
\end{lemma}

\begin{proof} The composition $Ab_C \mathcal{G}^n$ is an additive functor between abelian categories. Moreover, it is a right adjoint hence left exact. The assertion follows from \cite[Proposition 4.15]{HMR}.$\square$
\end{proof}

\begin{lemma}\label{nil3}Let $\mathcal{A} $ be an abelian category and $M,N\in U$. Suppose that the group $G$ acts nilpotently on $M$ and $N$.
Then the  $G$-action given by conjugation   on $Hom_U (M,N)$ is nilpotent. Moreover,  \[nil_G (Hom_{\mathcal{A}} (M,N))
\leq 2s\]  with $s=$ max$\{nil_G (M), nil_G (N)\}$.
\end{lemma}
\begin{proof} Let $m\in M$, $g\in G$ and $f\in Hom_{\mathcal{A}} (M,N)$.
Write \[[g,f](m)= \]\[g^{-1}fg(m)-f(m) =\] \[(g^{-1}fg(m) -fg(m)) + (fg(m) - f(m))=\]
\[[g^{-1},fg(m)] + f[g,m]
=\]
\[f_1 (m) + f_2(m)\] where $f_1 (m)$, $f_2 (m)$ are the evaluations at $m$ of an element in $Hom_{\mathcal{A}} (M,\Gamma^1 N)$, $Hom(_U (\Gamma^1 M, N)$ respectively.  The assertion follows by iteration of this argument
 since the summands which show up in an $2s$-fold commutator are in  $Hom_{\mathcal{A}} (\Gamma^i M,\Gamma^j  N)$ with $i+j=2s$.$\square$  
\end{proof}




\begin{theorem}\label{Ha} Let $X$ be an $n$-stage for $C$ and $M\in\mathcal{V}C$. If  $G\subset \pi_0 Aut^h (X))$ acts nilpotently on $C$ and $M$ then it acts nilpotently on each $AQ^m_C(X,M)$.
\end{theorem}

\begin{proof} We consider the spectral sequence  in \ref{SS}. 
The group $G$ acts on the spectral sequence and it acts nilpotently on the $E_1$ term by \ref{nil2} and \ref{nil3}. By strong convergence  it acts nilpotently on the $E_{\infty}$ term where we use again \cite[Proposition 4.15]{HMR}. A filtration argument shows that $G$ nilpotently on  $AQ^m_C(X,M)$.
\end{proof}

Let $X$ be a marked $n$-stage for $C$, $Y$ the induced marked $(n-1)$-stage and   $G\subset \pi_0 Aut^h_{\sharp}(X)$ a subgroup.
The induced  homomorphism
\begin{equation}\label{99}\pi_* (Aut^h_{\sharp}(X))\to \pi_* (Aut^h_{\sharp}(Y))
\end{equation} commutes with the $G$-action via conjugation on source and target.\\

Let $X$ be a marked $n$-stage for $C$, $Y$ the induced marked $(n-1)$-stage and   $G\subset \pi_0 Aut^h_{\sharp}(X)$ a subgroup.
 The group $G$ acts on $\pi_* (Aut^h_{\sharp}(X))$ and $\pi_* (Aut^h_{\sharp}(Y))$ via conjugation 
The induced  homomorphism
\begin{equation}\label{99}\pi_* (Aut^h_{\sharp}(X))\to \pi_* (Aut^h_{\sharp}(Y))
\end{equation} commutes with the $G$-action via conjugation on source and target.
 On $AQ^{*}_C (C,C[n])$ the induced $G$-action is

 On $AQ^{*}_C (C,C[n])$ the induced $G$-action is given by composition with the homomorphism induced by $g\in G$ and precomposition with the self map of $K_C (M,n)$ induced by $g^{-1}$ on $C[n+1]$. However the later is the identity by \ref{nil1}.
\begin{proposition}\label{lo}  The  long exact sequence in \ref{long} is one of $G$-equivariant homomorphisms.
\end{proposition}
\begin{proof} The homomorphism $Aut^h_{\sharp}(X))\to   Aut^h_{\sharp}(Y)$ is $G$-equivariant and the base point i.e. the identity is fixed by $G$.$\square$
\end{proof}
\section{Proof of Theorem \ref{main}}
We begin with the
  \begin{proof}of   \ref{maincor}.
Let $X$ and $G$ be as in \ref{main}.
\[H = \pi_0 (Aut^h_{\sharp}(X))\cap G .\]
Since $G$ acts nilpotently on $\pi^0(X)\cong C$ its image $\tilde{G}$ under\\
 $ \pi_0 Aut^h(X)\to Aut(C))$
  is  nilpotent by \ref{Hall}.\\
Consider the exact sequence of $G$-groups where $G$ acts  by conjugation
\begin{equation}1\to H \to G\to  \tilde{G}\to 1 \end{equation}

It suffices to show that $G$ acts nilpotently on $\pi_0 (Aut^h _{\sharp}(X))$.  Since if this is known, then $H\subset \pi_0 (Aut^h _{\sharp}(X))$ being a $G$-subgroub is $G$-nilpotent as well. Then also $G$, as an extension of nilpotent $G$-groups, is nilpotent.\\

Let $i \leq n$ and  $X(i)=sk_i (X)$ be the $(i-1)$-stage for $C$ the image of $X$ under the derived  $i$-skeleton functor.\\ 
Consider the tower of fibrations
\begin{equation}\label{tower} Aut^h _{\sharp}(X(n)))\to Aut^h _{\sharp}(X(n-1))\to \ldots Aut^h _{\sharp}(X(0))=1
\end{equation}
where the last equality holds by \ref{0stage}.
We show by induction that $G$ acts nilpotently on $\pi_0 (Aut^h _{\sharp}(X(i)))$. There is nothing to prove for $i=0$ and 
an application of \ref{Ha} and \ref{lo} closes the induction.

That $G$ is $\mathbb{F}$-complete is seen from consideration
of the exact sequence of nilpotent groups 
\begin{equation}1\to G_{\sharp}\to G \to \tilde{G}\to 1 \end{equation}
The group $\tilde{G}$ is $\mathbb{F}$-complete by assumption
and $ G_{\sharp}$ is the fundamental group of a component of the space $\mathcal{M}^{'}_{n} (C)$ which is  $\mathbb{F}$-complete by  \ref{comM}. 
 An application of  \cite[Corllary 2.3. b)]{Mo} shows that $G$ is $\mathbb{F}$-complete as well.$\square$\\
 \end{proof}
 
 \begin{proof}of   \ref{main}

We show that the homotopy groups $BAut_G (X)$ degree $\geq 2$  are nilpotent $G$-modules. Since $Aut_G^h (X)$ is group like it is enough to consider the component of the identity $Aut_1^h (X)$. We do an induction on $n$. For $n=0$ the space $Aut_1^h (X)$ is contractible. Consider the $s$-fold looped tower of fibration sequences \ref{tower}
 The homotopy fiber at level $k$ is   
 \begin{equation}\label{klooptower}\Omega^s \mathcal{A}^{n}(C,C[n-k]) \simeq \mathcal{A}^{n-s}(C,C[n-k])
 \end{equation}
 Again we conclude from \ref{Ha} and \ref{lo}.
 
The space  $Aut_1^h (X)$ is a component of the space $\Omega BAut^h_{\sharp}(X)$ which in turn is the loop space of a component of $\mathcal{M}^{'}_{n} (C)$ which is 
  $\mathbb{F}$-complete by  \ref{comM} below. Hence, the higher  homotopy groups  are $\mathbb{F}$-complete. 
 Now we conclude from \cite[VI.5.4.]{BK}$\square$
 
 \end{proof}

\begin{lemma}\label{framedcomplete}The space $\mathcal{M}^{f}_n (C)$ is $\mathbb{F}$-complete.
\end{lemma}
\begin{proof} We do an induction on $n$. The case $n=0$ follows from 
the fact that $\mathcal{M}^{f}_0 (C)\simeq Aut_C ([1])$ is discrete
together with the commutation of $\mathbb{F}$-completion with disjoint union \cite[7.7.1.]{BK}. For the inductive step we use \ref{framed}.
 The homotopy groups of the loop space $\mathcal{AQ}_C^{n+2}(C,C[n])$ are 
 $\mathbb{F}$-vector spaces which implies its $\mathbb{F}$- completeness. By  \cite[12.9.]{Bou2}
  a homotopy limit of $\mathbb{F}$-complete spaces  is  $\mathbb{F}$-complete.  The claim follows.$\square$ 
\end{proof}

\begin{lemma}\label{product} There is an equivalence 
\[\mathcal{M}^{f}_n (C)\simeq \mathcal{M}^{'}_n (C)\times Aut_C (C[n+1]).\]
\end{lemma}
\begin{proof} We claim that the morphism
\[\mathcal{M}'_n(C) \longrightarrow \mathcal{M}_n(C) \to B  Aut_C(C[n+1])\] is homotopic to a constant map on each connected component 
$B Aut^h_{\sharp} (X)$  given by a marked $n$-stage $X$.
 Since
 $Aut_C (C[n+1])$ is discrete  the induced map factors as 
\[Aut^h_{\sharp} (X)\to \pi_0 (Aut^h_{\sharp} (X))\to Aut_C(C[n+1])\]
and the later homomorphism  is  constant by \ref{nil1}.
$\square$
\end{proof}

\begin{lemma}\label{comM} The space $\mathcal{M}^{'}_{n} (C)$ is $\mathbb{F}$-complete.\end{lemma}
\begin{proof} By \ref{product} and \ref{framedcomplete}  $\mathcal{M}^{'}_{n} (C)$ is a subset of pathcomponents of the $\mathbb{F}$-complete space $\mathcal{M}^{f}_{n} (C)$.$\square$ 
\end{proof}

\section{Appendix}

The aim of this section is to give a proof of \ref{adjoint}.
In view of applications in later work we prove a general $\infty$-categorical version.\\
Under the Dold-Kan equivalence the model category $cnVec$ corresponds to the standard model category on cochain complexes of graded vector spaces $coch(nVec)$ in which the cofibrations are the maps which are monomorphisms in positive degrees and the fibrations are the surjections. All objects of $cnVec$ are cofibrant and fibrant. The inclusion functor of $nVec$ and the functor $\pi^* (-)$ are inverse equivalences between the homotopy category $Ho(cnVec)$ which identifies with  $nVec$. So the functor $\mathcal{G}$ is the functor induced by $G$ on $Ho(cnVec)$. An cochain complex  is called perfect if it is isomorphic in $Ho(cnVec)$ to a complex whose total dimension is finite.   It is well known that the compact objects in  $Ho(cnVec)$ are the perfect complexes.

The following result is a variant of the fundamental theorem of coalgebras for unstable coalgebras. Undoubtedly it is well known but we could not spot it in the literature.

\begin{theorem}\label{FT} Let $C\in \mathcal{CA}$ and $M\subset C$ a finite dimensional
sub vector space contained in degrees less than or equal to $n$. There is a finite dimensional unstable subcoalgebra $D\subset C$ with $M\subset D$.
\end{theorem}
\begin{proof} Let $C(1)$ be a finite subcoalgebra of $C$ which contains $M$. The proof given in \cite{Mi} for the fundamental theorem of coalgebras shows that $C(1)$ is generated in degrees less than or equal to $n$ as a graded vector space.
 Let $M(2)$ be the unstable right module generated by
 $C(1)$. It is finite as well. Note that all  elements in $M(2)/C(1)$ are of degree less than $n$. Repeat the construction for the finite submodule $M(2)$ and go on.
 After n steps we end up with an unstable coalgebra.$\square$

\end{proof}

\begin{lemma}\label{cofin} Let $D\in c\mathcal{CA}$ and $M\subset D_m^n$ a finite sub vector space. Then there is a   $K\in c\mathcal{CA}$ such that $K^l_p$ is of finite dimension for all $(l,p)$  and such  that  $M\subset K \subset D$. Moreover, $\pi^* (K) $ is trivial for $*>n$.
\end{lemma}
\begin{proof}
 By \ref{FT}  there exists a finite dimensional $H\in \mathcal{CA}$ with $M\subset H\subset D^n$.
Construct an $n$-truncated cosimplicial object $X$ defined by the sum of the images $\phi (H)$ of all cosimplicial operators  starting in degree $n$ and ending in degrees $\leq n$. Note that colimits are formed in underlying vector spaces so the resulting object is dergeewise finite. Apply the left adjoint $j$ to the forget functor from cosimplicail objects to $n$-truncated ones. Its explicit description
as a colimit shows that the resulting cosimplicial object is again finite in each degree.  The sum of the unstable coalgebras constructed for  a base of $M$ does the job. The last assertion is a consequence of the fact that the skeleton functor $sk_n$ is the good truncation at level $n$ of the category $cnV$.$\square$
\end{proof}
\vspace{0.5cm}
Recall that a functor between locally presentable categories is called finitary if it preserves filtered colimits. An object $A$ in a locally presentable category $\mathcal{A}$ is called compact or finitely presentable if the functor
\[\mathcal{A}\to Set,\hspace{0.5cm} B\to Hom_{\mathcal{A}}(A,B)\]  corepresented by $A$ is finitary. The object is called finitely generated if the functor  $Hom_{\mathcal{A}}(A,-)$ commutes with filtered colimits of monomorphisms \cite{A}.
\\

\begin{theorem} The functor $\mathcal{G}:nVec\to nVec$ is finitary.
\end{theorem}
\begin{proof} Let $V\in nVec$. By \cite[2.2.(5),2.2.(6),3.11.(2), 3.13.]{A} it is enough to show that given any monomorphism  \[i:M_0 \to G V\] with $M_0$ compact there is a compact $M$ and a monomorphism  \[j:M\to V\] such that $i$ factors through $Gj$. The object $M_0$ is $m$-truncated for some $m$. So we may assume that it is of finite  total 
dimension.
 Now let $ C_0 $ be a compact
 $G$  subcoalgebra containing $M_0$. This exists by \ref{cofin} Let $k: C_0\to V$ be the morphism in $cnV$ adjoint to the inclusion homomorphism. Factor it 
\[C_0 \stackrel{u}{\to} M \stackrel{v}{\to} V\] into an epimorphism $u$  followed by a monomorphism $v$. Then $M$ is finite and the composition 
\[M_0 \to C_0  \to GC_0 \to G M \to G V\]
gives the searched for factorization.$\square$
\end{proof}

In the following we will freely use the language of $\infty$-categories (see \cite{L1} and \cite{L2}).
An $\infty $-category $\mathcal{D}$ is called accessible if it is the $\kappa$-Ind-completion of its full subcategory 
$\mathcal{D}^{\kappa}\subset \mathcal{D}$ of $\kappa$-compact objects for a regualr cardinal $\kappa$ (see \cite{L1} for details). In case $\mathcal{D}$ admits in addition small colimits it is called presentable. We call a a functor $F:\mathcal{C}\to \mathcal{D}$ between 
$\infty $-categories accessible if it preserves $\kappa$-filtered colimits.

The next result is to be found in \cite[6.84]{He} (see \cite{GKR} for the 1-categorical version).

\begin{theorem} Let $\mathcal{D}$ be a presentable $\infty$-category and $C$ an accessible comonad on  $\mathcal{D}$. Then the category of of $C$-coalgebras $Coalg_C$ is presentable.\end{theorem}


We write $\widehat{\mathbf{Cat}}_{\infty}$ for the $\infty$-category of not necessarily small  $\infty$-categories.
Let $\mathcal{S}$ be the $\infty$-category of spaces i.e. simplicial sets and   $Gr$ the $\infty$-category of group like $E_{\infty}$- spaces. Finally,  let  $\nu :Gr\to \mathcal{S}$ be the forget functor.
\begin{definition} An abelian group object in an $\infty$-category $\mathcal{D}$ with finite products  is a functor
 $\phi :\mathcal{D}^{op}\to Gr$
  such that the composition
   \[\mathcal{D}^{op}\stackrel{\phi}{\to} Gr
    \stackrel{\nu }{\to}\mathcal{S}\] is representable.
Denote by $Gr(\mathcal{D})\subset Fun(\mathcal{D}^{op},Gr)$ the full subcategory of abelian group objects.\\
Dually, in case $\mathcal{D}$ has finite coproducts, let $coGr(\mathcal{D}):=Gr(\mathcal{D}^{op})^{op}\subset Fun(\mathcal{D}, Gr)^{op}$ be the $\infty$-category of coabelian cogroup objects in $\mathcal{D}$.
\end{definition}

\begin{remark} One has $Gr \simeq Gr(\mathcal{S})$.\end{remark}

For $\infty$-categories $\mathcal{C}, \mathcal{D}$ let $Fun^R (\mathcal{C}, \mathcal{D})$ respectively 
$Fun^L (\mathcal{C}, \mathcal{D})$ denote the full subcategories of $Fun (\mathcal{C},\mathcal{D})$ which are right adjoints receptively left adjoints. There is a canonical equivalence 
\[Fun^L (\mathcal{C}, \mathcal{D})\simeq Fun^R (\mathcal{D}, \mathcal{C})^{op}\] \cite[Proposition 5.2.6.2.]{L1} which sends a left adjoint to its right adjoint. Moreover, by \cite[Theorem 5.1.5.6.]{L1} one has 
\[\mathcal{D}\simeq Fun^l (\mathcal{S},\mathcal{D})\]
and hence
\[\mathcal{D}\simeq Fun^l (\mathcal{S},\mathcal{D})\simeq Fun^R (\mathcal{D}, \mathcal{S})^{op}.\]
Suppose that $\mathcal{C}, \mathcal{D}$ have finite coproducts. Let $Fun^{\sqcup} (\mathcal{C}, \mathcal{D})$ be the category of functors which preserve finite coproducts.

 The next theorem and its proof was communicated to me by  Hadrian Heine. 



\begin{theorem} Let  $\mathcal{D}$ be a presentable $\infty$-category. Then $coGr(\mathcal{D})$ is presentable.
\end{theorem}
\begin{proof} First we show that $coGr(\mathcal{D})$ has small colimits which are preserved by the forget functor to $\mathcal{D}$.
This follows from the dual assertion about an $\infty$-category $\mathcal{B}$ with small limits. In fact, the $\infty$-category 
$Gr$ has small limits which are preserved by the forget functor to $\mathcal{S}$. Consequently, the category $Fun(\mathcal{B}^{op},Gr )$ has small limits which are formed object wise and are preserved by the forget functor to $Fun(\mathcal{B}^{op},\mathcal{S} )$. Because the representable presheaves are closed under limits it follows that $Gr(\mathcal{B})$ is closed under small limits in $Fun(\mathcal{B}^{op}, Gr)$.\\
So we are left to show that $coGr(\mathcal{D})$ is accessible. Denote the  $\infty$-category of finitely generated free abelian group objects by $\mathcal{F}$ and let $Fun^{\sqcup}(\mathcal{F}, \mathcal{D})$ the category of functors which preserve finite coproducts. By \ref{equi} there is a canonical equivalence
\[coGr(\mathcal{D})\simeq Fun^{\sqcup}(\mathcal{F}, \mathcal{D}).\]
But the later is accessible by \ref{isacc}.$\square$
\end{proof}

\begin{corollary}\label{adjoint2}
Let  $C$ be a presentable $\infty$-category and  $\mathcal{D}$
an accessible comonad on $C$. Then $coGr(Coalg_C)$ is presentable.
\end{corollary}

\begin{proof} (of  \ref{adjoint}). This is immediate from \ref{adjoint2}.$\square$
\end{proof}

\begin{lemma}\label{equi} There is a canonical equivalence
\[\theta :Gr(\mathcal{D})\simeq Fun^{R}(\mathcal{D}, \mathcal{S})^{op}.\]
\end{lemma} 

\begin{proof} There is an equivalence 
$Fun^L (\mathcal{S},\mathcal{D})\simeq \mathcal{D}$ by \cite[Proposition 5.1.5.6.]{L1}. This together with the equivalence $Fun^L (\mathcal{S},\mathcal{D})\simeq Fun^{R}(\mathcal{D}, \mathcal{S})^{op}$ shows the assertion.$\square$
\end{proof}

\begin{lemma}\label{equi0} There is a canonical equivalence
\[coGr(\mathcal{D})\simeq Fun^{\sqcup}(\mathcal{F}, \mathcal{D}).\]
\end{lemma}
\begin{proof} For a $\infty$-category $\mathcal{B}$ with finite coproducts let $\mathcal{P}_{\Sigma}(\mathcal{B})\subset Fun(\mathcal{B}^{op},\mathcal{S})$ the full subcategory of functors which preserve finite products. According to \cite[5.5.8.15.]{L1}
the  Yoneda embedding induces an equivalence of categories \[Fun^L(\mathcal{P}_{\Sigma}(\mathcal{B}),\mathcal{D})\simeq Fun(\mathcal{B},\mathcal{D}). \] As a consequence, the embedding $\mathcal{F}\subset Gr$ extends to a functor $\theta :\mathcal{P}_{\Sigma}(\mathcal{F})\to Gr$. We claim that $\theta $ is an equivalence.
By \cite[Proposition 5.5.8.22]{L1} it is sufficient to show the following :\\
(1) For each $X\in \mathcal{F}$, the
  functor $Gr(-,X):Gr:\to \mathcal{S}$ preserves sifted colimits;\\
  (2) $Gr$ is generated under sifted colimits by $\mathcal{F}$.\\
  Because $X$ is free on a finite set $\{1,\ldots ,n \}$ there is an equivalence $Gr(X,-)\simeq \mathcal{S}(\{1,\ldots ,n \},-)\circ \nu $. Now  the functors $\nu$ and $\mathcal{S}(\{1,\ldots ,n \},-)\simeq id^{\{1,\ldots ,n \}}$ both preserve sifted colimits. The former by \cite[Proposition 3.2.3.1.]{L2} the later because finite products commute with sifted colimits. This gives assertion (1).
  For assertion (2) note that $Gr$ is generated by free group like $E_{\infty}$-spaces under sifted colimits. These  in turn are generated by $\mathcal{F}$ under sifted colimits  because $\mathcal{S}$ is generated by sifted colimits of the full subcategory of finite sets \cite[Proposition 5.5.8.13.]{L1}.\\
 Hence we find 
 \[Fun^L (Gr,\mathcal{D})\simeq Fun^{\sqcup}(\mathcal{F}, \mathcal{D}).\]
 The canonical equivalence 
 $Gr(Fun(\mathcal{D},\mathcal{S}\simeq Fun(\mathcal{D},Gr)$ restricts to an equivalence $Gr(Fun^R (\mathcal{D},\mathcal{S})\simeq Fun^R (\mathcal{D},Gr)$.
 From this together with \ref{equi0} we get
 \[coGr(\mathcal{D})\simeq Gr(Fun^R (\mathcal{D},\mathcal{S})\simeq Fun^R (\mathcal{D},Gr)\simeq Fun^L (Gr,\mathcal{D})\simeq Fun^{\sqcup}(\mathcal{F}, \mathcal{D})\] which finishes the proof.$\square$
 \end{proof}
 In the following we make use of the concept of symmetric monoidal $\infty$-category in the form developed in \cite[2.4.2.]{L2}.

\begin{definition} Let $\mathcal{F}in_* $ denote the category of finite sets with base point. For $n\geq 0$ let $\langle n \rangle = *\sqcup \{1,\ldots n\}$ be the points set which is obtained from the $n$ element set by adding a base point $*$. For each $n\geq0$ let $\tau_1 :\langle n \rangle \to \langle 1 \rangle $ be the map which sends the base point to itself and the rest to $1$.
\end{definition}
\begin{definition}
 A symmetric monoidal  $\infty$-category is a functor \[F:\mathcal{F}in_* \to \widehat{\mathbf{Cat}}_{\infty}\]
such that $F(\langle 0 \rangle)$ is contractible and such that the maps $\tau_1$ induce an equivalence \[F(\langle n \rangle)\simeq F(\langle 1 \rangle)^{\times n}\] for every $n\geq 1$.
The underlying $\infty$-category of $F$ is $F(\langle 1 \rangle)$.
  \end{definition}
  
  We write $Acc\subset \widehat{\mathbf{Cat}}_{\infty}$ for the subcategory of accessible $\infty$-categories and accessible functors.
 A symmetric  $\infty$-category $F$ is called accessible if $F$ factors through $Acc$. For two symmetric monoidal $\infty$-categories $\mathcal{C}, \mathcal{D}$ the  $\infty$-categories
 of symmetric monoidal functors $Fun^{\otimes}(\mathcal{C}, \mathcal{D})$ satisfies the following universal property:
 For each $K\in \widehat{\mathbf{Cat}}_{\infty}$ there is a natural equivalence
 \[\widehat{\mathbf{Cat}}_{\infty}(K, Fun^{\otimes}(\mathcal{C}, \mathcal{D})\simeq Fun(\mathcal{F}in_* ,\widehat{\mathbf{Cat}}_{\infty})(\mathcal{C}, \mathcal{D}^K)\]
 where $\mathcal{D}^K)$ is given by the composition 
 \[\mathcal{F}in_* \stackrel{\mathcal{D}}{\to}\widehat{\mathbf{Cat}}_{\infty}\stackrel {Fun(K,-)}{\longrightarrow}\widehat{\mathbf{Cat}}_{\infty}\]
 
 \begin{theorem} Let $\mathcal{B}$ and $\mathcal{C}$ be a  symmetric monoidal $\infty$-categories where the former is essentially small and the later is accessible. Then  $Fun^{\otimes}(\mathcal{B}, \mathcal{C})$ is accessible.
 \end{theorem}
 
 \begin{proof} We use that mapping spaces in functor categories can be described as ends. In fact by \cite[Proposition 5.1.]{GHN} there is, for two functors $F,G:\mathcal{X}\to \mathcal{Y}$ between  $\infty$-categories,  a canonical equivalence \[Fun(\mathcal{X},\mathcal{Y})(F,G)\simeq \int_{X\in \mathcal{X}}\mathcal{Y}(F(X),G(X))\]
 From this get an equivalence
 \[Fun(\mathcal{F}in_* ,\widehat{\mathbf{Cat}}_{\infty})(\mathcal{B}, \mathcal{C})\simeq \int_{\langle n \rangle \in \mathcal{F}in_*}\widehat{\mathbf{Cat}}_{\infty}(\mathcal{B}^{\times n},\mathcal{C}^{\times n})\]
 Hence we have equivalences for each $K\in 
\widehat{\mathbf{Cat}}_{\infty}$
 \[\widehat{\mathbf{Cat}}_{\infty}(K, \int_{\langle n \rangle }Fun(\mathcal{B}^{\times n},\mathcal{C}^{\times n}))\simeq 
 \int_{\langle n \rangle }\widehat{\mathbf{Cat}}_{\infty}(K,Fun(\mathcal{B}^{\times n},\mathcal{C}^{\times n}))\simeq\]
 \[ \int_{\langle n \rangle }\widehat{\mathbf{Cat}}_{\infty}(\mathcal{B}^{\times n},Fun(K,\mathcal{C}^{\times n}))\simeq Fun(\mathcal{F}in_* ,\widehat{\mathbf{Cat}}_{\infty})(\mathcal{B}, \mathcal{C}^K )\simeq
 \widehat{\mathbf{Cat}}_{\infty}(K, Fun^{\otimes}(\mathcal{B},\mathcal{C})\]
 
 which represents an equivalence 
 
 \[\int_{\langle n \rangle }Fun(\mathcal{B}^{\times n},\mathcal{C}^{\times n})\simeq Fun^{\otimes}(\mathcal{B},\mathcal{C})\]
 Since by assumption $\mathcal{B}$ is essentially small and $\mathcal{C}$ is accessible an application of \cite[Proposition 5.4.4.3]{L1} gives that $\int_{\langle n \rangle }Fun(\mathcal{B}^{\times n},\mathcal{C}^{\times n})$ is accessible for all $n$. All transition morphisms in the diagram of the end are in $Acc$.
The $\infty$-category $Acc$ has small limits which are preserved by the inclusion $Acc\subset \widehat{\mathbf{Cat}}_{\infty}$ by \cite[Proposition 5.4.7.3]{L1}. Consequently the end is in $Acc$ as well.$\square$ 
 \end{proof}
 
 \begin{corollary}\label{isacc}  Let $\mathcal{B}$ and $\mathcal{C}$ be a  symmetric monoidal $\infty$-categories where the former is essentially small and the later is accessible. Assume that
 $\mathcal{B}$ and $\mathcal{C}$ have finite products. Then 
 the $\infty $-category $Fun^{\sqcup}(\mathcal{B},\mathcal{C})$ is accessible.
 \end{corollary}
 \begin{proof}It follows from \cite[2.4.3.8.]{L2}
 that there is an equivalence \[Fun^{\sqcup}(\mathcal{B},\mathcal{C})\simeq Fun^{\otimes}(\mathcal{B},\mathcal{C})\] when $\mathcal{B},\mathcal{C}$ carry the cocartesian symmetric momoidal structure.$\square$
 \end{proof}
 
 Now as we have seen that $coGr(\mathcal{D})$ is presentable for presentable $\mathcal{D}$ and hence that the forget functor $\phi$ admits a right adjoint. We may ask if the forget functor $\phi $ is comonadic.
This is answered affirmatively below.  The notion of a $\phi$-split object can be found in\cite[Definition 4.7.2.2.]{L2}.
 
 \begin{proposition} Let $\mathcal{D}$ be a presentable $\infty$-category.  The forget functor  \[\phi:coGr(\mathcal{D})\to \mathcal{D}\] is comonadic.$\square$
 \end{proposition}
 
 \begin{proof}
   According to the dual of the Barr-Beck theorem of Lurie \cite[Theorem 4.7.3.5.]{L2} it is enough to show that every cosimplicial $\phi$-split object in $coGr(\mathcal{D})$ has a totalization which is preserved by $\phi$ and that the forget functor to $\mathcal{D}$ is conservative. To see this we may embed 
 $\mathcal{D}$ fully  in  a category  $\mathcal{E}$ if the  embedding preserves finite coproducts. This is so since in this situation  $\phi$ is the pullback of the forget functor $\psi:coGr(\mathcal{E})\to \mathcal{E}$ along $\mathcal{D}\subset \mathcal{E}$ and the property in question is stable under pullback.\\
  So we may choose for $\mathcal{E}$ copresheaves on $\mathcal{D}$ i.e. $\mathcal{E}=\mathcal{P}(\mathcal{D}^{op})^{op}$ . In this case the dual assertion is true  for $\mathcal{E}^{op}$  by the following argument.
   Equipped with the cartesian structure  $\mathcal{E}^{op}$ is a presentable symmetric monoidal $\infty$-category such that the tensor product preserves sifted colimits. By \cite[Proposition 3.2.3.1.]{L2} it follows that 
 the $\infty$-category
 $Cmon( \mathcal{E}^{op})$ of commutative monoids in  $\mathcal{E}^{op}$ admits sifted colimits which are preserved by the forget functor to  $\mathcal{E}$. Now $Gr(\mathcal{E}^{op})$ is in  $Cmon( \mathcal{E}^{op})$ closed under sifted colimits since products preserve sifted colimits  as noted. 
The forget functor $Gr( \mathcal{E}) \to Cmon( \mathcal{E})$ is conservative by fullness. The  forget functor $Cmon( \mathcal{E})\to \mathcal{E}$ is conservative as well and $Gr( \mathcal{E})\to  \mathcal{E}$ factors over $Cmon( \mathcal{E})$ and so it is conservative.  By Barr-Beck $Gr(\mathcal{E}^{op}))\to \mathcal{E}^{op}$ is monadic.
In fact all realizations are preserved not only the split ones. Consequently, by what was said above,  $\phi$ is comonadic.$\square$
\end{proof}
\begin{corollary}Let $\mathcal{D}$ be a presentable $\infty$-category and $C$ an accessible comonad on  $\mathcal{D}$. Then the forget functor $\phi :coGr(Coalg_C)\to (Coalg_C$ is comonadic.
\end{corollary}
\begin{remark} The application which we make of the results in the appendix is $1$-categorical. But a possible direct proof would in any case involve $2$-categorical arguments as in \cite{GKR}. On the other hand we have $\infty$-categorical applications for Andr\'{e}-Quillen  cohomology of
$E_{\infty}$-coalgebras in mind.
\end{remark}

\noindent
 Manfred Stelzer,
Universit\"at Osnabr\"uck.

\end{document}